\documentclass[11pt]{article}
\usepackage{fullpage}
\usepackage{epsfig}
\usepackage{graphics}
\usepackage{latexsym}
\usepackage{amsmath}
\usepackage{amsfonts}
\usepackage{amssymb}
\usepackage{mathrsfs}
\usepackage{amsthm}
\usepackage{xspace}
\usepackage{epstopdf}
\usepackage{float}
\usepackage{hyperref}
\usepackage{mathtools}

\numberwithin{equation}{section}
\usepackage[ruled,vlined]{algorithm2e}
\SetArgSty{textrm}

\usepackage[english]{babel}
\usepackage[nottoc]{tocbibind}

\usepackage{caption}
\usepackage{subcaption}

\usepackage[usenames,dvipsnames]{color}
\usepackage{pifont}
\usepackage{booktabs}
\usepackage{multirow}

\newcommand{\redcomment}[1]{\textcolor{red}{\textrm{#1}}}

\usepackage{amsthm}     
\theoremstyle{plain}     
\newtheorem{theorem}{Theorem}

\newtheorem{corollary}{Corollary}

\newtheorem{lemma}{Lemma}
\newtheorem{proposition}{Proposition}
\theoremstyle{definition} 

\newtheorem{observation}{Observation}

\theoremstyle{remark} 

\makeatletter
\@addtoreset{equation}{section}
\def\section{\@startsection {section}{1}{\z@}{-3.5ex plus -1ex minus
 -.2ex}{2.3ex plus .2ex}{\large\bf}}
\makeatother

\def\bfm#1{\mbox{\boldmath$#1$}}

\def\u{\bfm u}
\def\0{\bfm 0}

\DeclareMathAlphabet{\mathpzc}{OT1}{pzc}{m}{it}

\newcounter{my}

\newcounter{my2}

\newcounter{my3}

\newcounter{my4}

\newcounter{my5}

\newcounter{my6}

\allowdisplaybreaks 

\begin{document}

\title{Randomized Max-Vertex-Coverage Interdiction under Matroid Constraints}


\date{}
\maketitle

\vspace{-3em}
\begin{center}

\author{Changjun Wang$^{1}$,\quad  Chenhao Wang $^{2,3}$\\
${}$\\
$1$ Academy of Mathematics and Systems Science, Chinese Academy of Sciences\\
$2$ Beijing Normal University-Zhuhai\\
$3$ Beijing Normal-Hong Kong Baptist University\\
\medskip
}

\end{center}

\begin{abstract}

We study a class of bilevel interdiction problems in which the follower's optimization problem is computationally intractable. Motivated by network defense applications, we introduce the Randomized Max-Vertex-Coverage Interdiction (RMVCI) problem under matroid constraints. In this zero-sum Stackelberg game, the leader commits to a randomized interdiction strategy over feasible vertex subsets, while the follower, after observing the induced protection probabilities, chooses a matroid-constrained attack to maximize the expected coverage of network edges.

The main challenge stems from the fact that the follower's problem is a matroid-constrained maximum vertex coverage problem and is therefore NP-hard. To address this difficulty, we first develop a general approximation framework for bilevel optimization problems with hard follower responses. The framework is based on replacing the follower's value function by a surrogate objective that approximates the follower's optimal payoff while preserving tractability of the leader's optimization problem.

For the RMVCI problem, we formulate the follower's problem as an integer linear program, establish a tight integrality gap of $4/3$ for its linear relaxation, and derive a polynomial-time $4/3$-approximation algorithm via pipage rounding. We then show that a carefully designed surrogate objective admits a marginal-probability reformulation that transforms the randomized interdiction problem into a tractable optimization problem over the leader's matroid polytope. This yields a polynomial-time $2$-approximation algorithm for RMVCI under general matroid constraints.
Beyond the specific application studied here, our results provide a new perspective on approximation methods for {general} bilevel optimization problems. 
\end{abstract}

\section{Introduction}

Bilevel optimization provides a powerful framework for modeling hierarchical decision-making processes involving two interacting decision makers. In a bilevel optimization problem, a leader first chooses a decision, after which a follower reacts optimally. The leader's objective depends on both decisions, resulting in a nested optimization structure that is often computationally challenging. Bilevel optimization has found numerous applications in transportation, pricing, facility location, energy systems, security planning, and network design; see the monographs of Dempe~\cite{dempe2002foundations} and the survey of Colson et al.~\cite{colson2007overview} for comprehensive overviews.

Among the many application domains of bilevel optimization, \emph{network interdiction} has emerged as one of the most extensively studied. In a network interdiction problem, a defender (leader) allocates limited defensive resources to protect or {interdict} 
a network, while an attacker (follower) subsequently chooses an action that maximizes the resulting damage. Classical examples include  maximum-flow interdiction~\cite{wollmer1964removing,wood1993deterministic}, shortest-path interdiction~\cite{fulkerson1977maximizing,israeli2002shortest}, spanning tree interdiction~\cite{chestnut2017interdicting,frederickson1999increasing,weninger2025interdiction}, matching interdiction~\cite{zenklusen2010matching}, facility coverage interdiction~\cite{church2004identifying}, {maximum-capacity-path interdiction~\cite{eor/TayyebiMS23},} and interdiction of packing linear programs~\cite{dinitz2013packing}. We refer the reader to the recent survey of Smith and Song~\cite{eor/SmithS20} for a broader overview of the area.

Most successful interdiction models share a common feature: the follower's optimization problem is either polynomially solvable or admits a tractable mathematical programming reformulation. This property is crucial because many classical bilevel optimization techniques (e.g., duality-based reformulations, value-function approaches, and KKT-based methods) rely fundamentally on the ability to characterize or solve the follower's problem efficiently. When the lower-level problem becomes NP-hard, however, these techniques often lose their effectiveness, as even evaluating the follower's optimal response may already be computationally intractable.


While several works have studied interdiction models
with NP-hard follower problems,
their approaches are largely divided into two categories.
The first develops exact algorithms based on
problem-specific reformulations and decomposition techniques \cite{lozano2017value,DBLP:journals/mp/CroceS20}. The second establishes approximation guarantees
for particular hard-follower models \cite{DBLP:journals/tcs/ChenZ13,chen2022approximation,linhares2017improved}. Neither direction provides a general framework
for designing approximation algorithms
for bilevel optimization problems with
computationally intractable lower levels.



Consequently, despite the substantial progress made on both exact and problem-specific approximation approaches, a significant methodological gap remains. To the best of our knowledge, there is currently no general approximation paradigm for network interdiction problems with NP-hard follower subproblems. Addressing this gap is particularly important because many natural attack models in network defense lead to hard combinatorial optimization problems at the lower level. The randomized max-vertex-coverage interdiction problem studied in this paper provides a representative example of such a setting and serves as a vehicle for developing a more general approximation methodology for bilevel optimization with hard follower responses.

\subsection{The problem and our results}

We study a new interdiction model, termed the \emph{Randomized Max-Vertex-Coverage Interdiction} (RMVCI) problem under matroid constraints. Let $G=(V,E)$ be an edge-weighted graph. The leader first commits to a randomized protection strategy by selecting a probability distribution over matroid-feasible subsets of vertices. After observing the induced protection probabilities, the follower chooses a matroid-feasible subset of vertices to attack. The follower's objective is to maximize the expected total weight of edges incident to vertices that are simultaneously attacked and left unprotected, whereas the leader seeks to minimize this quantity.

The RMVCI problem combines two sources of complexity.
The follower's optimization problem generalizes the
matroid-constrained maximum vertex coverage problem  \cite{HS23}
and is therefore NP-hard,
while the leader's decision space consists of probability
distributions over exponentially many feasible interdiction sets.
These two difficulties interact through the follower's expected payoff,
rendering both exact optimization and approximation highly nontrivial.
These features make RMVCI a particularly appealing testbed for developing approximation methodologies for bilevel optimization with computationally hard follower responses.



The primary goal of this paper is therefore not merely to study the RMVCI problem itself, {but to develop a broader approach for approximating bilevel optimization problems, particularly those for which the follower's problem is difficult to optimize directly.} 
Our central idea is to replace the follower's exact value function by a carefully designed surrogate objective. Such a surrogate should satisfy two competing requirements: it must approximate the follower's {optimal value} 
within a provable factor, while simultaneously exposing enough structure to make the leader's optimization problem tractable.

To pursue this approach, we begin by studying the follower's optimization problem in isolation. We formulate it as an integer linear program  and analyze its LP relaxation. In particular, we establish a tight integrality gap of $4/3$, which yields a polynomial-time $4/3$-approximation algorithm for the follower. While this analysis provides valuable information about the structure of the lower-level problem, it also reveals that directly employing the LP relaxation as a surrogate within the bilevel framework does not lead to an efficiently solvable leader problem.

This observation motivates the search for a different surrogate objective. We develop a surrogate that captures the essential structure of the follower's payoff while admitting a significantly simpler representation. A key insight is that the resulting surrogate depends only on the marginal interdiction probabilities induced by the leader's randomized strategy. Exploiting this property, we derive a reformulation over the leader's matroid polytope, thereby avoiding explicit optimization over exponentially many mixed strategies. Combining these ingredients yields a polynomial-time $2$-approximation algorithm for the RMVCI problem under general matroid constraints.

\paragraph{Our contributions.}
The main contributions of this paper are summarized as follows.

\begin{itemize}

\item We introduce the Randomized Max-Vertex-Coverage Interdiction problem under matroid constraints, a bilevel interdiction model with an NP-hard follower response.

\item We propose a general approximation framework for bilevel optimization problems with computationally intractable lower levels based on surrogate value-function approximations.

\item We formulate the follower's problem as an integer linear program, prove that its natural LP relaxation has a tight integrality gap of $4/3$, and derive a matching polynomial-time approximation algorithm.

\item We develop a new surrogate objective for the leader's problem and show that it admits a reformulation in terms of marginal interdiction probabilities over the leader's matroid polytope. 

\item By combining these ingredients, we obtain a polynomial-time $2$-approximation algorithm for RMVCI under general matroid constraints.

\end{itemize}

More broadly, our results {on RMVCI} demonstrate that surrogate value-function approximations may provide a useful paradigm for designing approximation algorithms for bilevel optimization problems, particularly when  {the follower's value function is difficult to optimize directly}. 

\paragraph{Organization.}
The remainder of the paper is organized as follows. {Section~\ref{sec:model} introduces the mathematical formulation of RMVCI.} Section~\ref{framework} presents a general approximation framework for bilevel optimization problems. Section~\ref{sec:inner} studies the follower's optimization problem and establishes the tight $4/3$ integrality gap. Section~\ref{sec:out} develops the surrogate reformulation and the 2-approximation algorithm for the leader. 

\subsection{{Related work}}

Bilevel optimization provides a natural modeling framework for hierarchical
decision-making problems in which a leader first commits to a decision and a
follower subsequently optimizes its own objective
\cite{dempe2002foundations,colson2007overview}. Network interdiction is one of
the most prominent classes of bilevel optimization problems. Exact algorithms
for interdiction problems have been widely studied; see, for example,
\cite{eor/SmithS20,lozano2017value}. Approximation algorithms have also been developed for a variety of bilevel and interdiction problems. For example,
Chestnut and Zenklusen \cite{chestnut2017interdicting} give approximation
results for interdicting structured combinatorial optimization problems with
\(\{0,1\}\)-objectives, while Dinitz and Gupta \cite{dinitz2013packing} study
packing interdiction and related partial covering problems. Chen and Zhang
\cite{DBLP:journals/tcs/ChenZ13} study approximation algorithms for the bilevel
knapsack problem. Other examples include maximum-flow interdiction
\cite{wollmer1964removing,wood1993deterministic}, shortest-path interdiction
\cite{fulkerson1977maximizing,israeli2002shortest}, and spanning-tree
interdiction \cite{frederickson1999increasing,weninger2025interdiction}.
These results provide important precedents for approximation in interdiction,
but they are often problem-specific: the approximation guarantee typically
relies on the particular combinatorial structure of the follower problem and
does not directly yield a general surrogate-value-function approach. 


\paragraph{Randomized and continuous interdiction.}
Randomization is an important theme in interdiction models. Bertsimas,
Nasrabadi, and Orlin \cite{bertsimas2016power} investigate the power of
randomization in network interdiction, and Holzmann and Smith
\cite{ior/HolzmannS21} study shortest-path interdiction with randomized
interdiction strategies. In these models, the leader's decision is no longer a
single interdiction set, but a distribution over possible interdiction actions.
This perspective is related to continuous interdiction, where the leader's
decision is represented by fractional or continuous resource allocations.
Carvalho, Lodi, and Marcotte \cite{DBLP:journals/orl/CarvalhoLM18} study a
continuous bilevel knapsack problem, while Tayyebi, Mitra, and Sefair
\cite{eor/TayyebiMS23} consider a continuous maximum-capacity-path interdiction
problem. Both randomized and continuous interdiction enlarge the leader's
decision space beyond deterministic discrete actions, but they lead to
different algorithmic challenges.
In the RMVCI model studied in this paper, the leader randomizes over
matroid-feasible vertex sets, and the follower then solves a
matroid-constrained coverage problem. Thus, the main challenge is not only to
optimize over a distribution with exponentially many possible supports, but
also to control the effect of this distribution on an NP-hard follower value
function.

\paragraph{Matroid interdiction and matroid-constrained optimization.}
Matroids provide a unifying abstraction for many independence systems in
combinatorial optimization. They have long played an important role in exact
and approximation algorithms, partly because matroid polytopes admit strong
structural descriptions and efficient optimization or separation procedures
\cite{grotschel1988geometric,cunningham1984testing}. A closely related line of
work is matroid interdiction. Since minimum spanning trees are minimum-weight
bases of graphic matroids, MST interdiction can be viewed as a special case of
matroid interdiction. Weninger and Fukasawa
\cite{weninger2025interdiction} study interdiction of minimum spanning trees
and other matroid bases, showing that matroid interdiction is NP-complete even
for uniform matroids and developing an exact branch-and-bound algorithm based
on dynamic-programming upper bounds. Related work also considers interdiction
under partition matroid structure \cite{ketkov2025class}.
These papers are closely related to the present work through their use of matroidal
structure at the leader and/or follower level, but they differ in a crucial
respect. In classical matroid-basis interdiction, the lower-level problem is a
minimum-weight basis problem and is solvable by a greedy algorithm. In contrast,
in RMVCI the follower faces a matroid-constrained maximum vertex coverage
problem, which is NP-hard. Thus, even evaluating the follower's exact response
is computationally difficult.

\paragraph{Coverage and submodularity.}
The follower's problem in RMVCI is closely connected to maximum coverage and
monotone submodular maximization under matroid constraints. The classical
approximation theory for submodular maximization originates with Nemhauser,
Wolsey, and Fisher \cite{nemhauser1978analysis} and Nemhauser and Wolsey
\cite{NemhauserW78}. Later developments include the multilinear-extension and
continuous-greedy framework of Calinescu et al.
\cite{calinescu2011maximizing}, as well as pipage
rounding \cite{ageev2004pipage}, which is particularly useful for rounding
fractional points in matroid polytopes while preserving approximation
guarantees. Recent work on matroid-constrained vertex-cover-type problems
\cite{HS23} is also relevant to the lower-level combinatorial structure. 

The present work lies at the intersection of these themes. It studies a
randomized network coverage interdiction problem under matroid constraints in which the
follower's exact value function is computationally intractable. The key
methodological distinction is to replace the follower value function by
a surrogate that simultaneously provides a provable approximation to the
follower optimal value and admits a tractable leader-side reformulation. 
This perspective complements existing
exact and problem-specific approximation approaches and provides a framework
for approximating bilevel interdiction problems.

\section{Preliminary and problem formulation}\label{sec:model}
\paragraph{Matroid.} 
A set system $\mathcal M= (U, \mathcal{I})$ on a ground set $U$ (where  $\mathcal I$ is a family of subsets of $U$) is called a \textit{matroid} if it satisfies the following three conditions: 1) $\varnothing\in\mathcal I$,
2) if $B \in \mathcal{I}$ and $A \subseteq B$, then $A \in \mathcal{I}$, and  
3) for any two sets $A, B \in \mathcal{I}$ with $|A| < |B|$, there exists an element $i \in B$ such that $A \cup \{i\} \in \mathcal{I}$.  
The sets in $\mathcal{I}$ are called \textit{independent sets.} 
The \emph{matroid polytope} of $\mathcal M$ is the convex hull of the incidence vectors of  independent sets, or equivalently \cite{edmonds1970}:
\[\mathcal P=\text{conv}\{x\in\{0,1\}^n:A(x)\in \mathcal I\} =\{x\ge 0: \sum_{i\in S}x_i\le r(S), \forall S\subseteq U\},\]
where $A(x)$ denotes the support set $\{i\in U:x_i=1\}$ of vector $x$, and  $r(S)=\max\{|I|:I\subseteq S,I\in\mathcal I\}$ is the rank function, i.e., the size of the largest independent subset of $S$.
As usual, we assume the existence of an \emph{independence oracle} for testing whether a set is independent. Given such an oracle, Cunningham \cite{cunningham1984testing} shows that the matroid polytope has an efficient \emph{separation oracle} that decides whether a given point lies inside the polytope and, if not, provides a separating hyperplane that excludes the point from the polytope.

\paragraph{RMVCI problem under matroid constraints.} 
Let $G=(V,E)$ be a graph with vertex set $V=\{1,\ldots,n\}$ and edge set $E$. Each edge $e=(u,v)\in E$ is associated with a weight $w_e$, and we denote by $e=uv$ for short.  For every subset $V'\subseteq V$ of vertices, let $\delta(V')\subseteq E$ be the set of edges covered by $V'$, that is, an edge $e=uv$ is in $\delta(V')$ if $\{u,v\}\cap V'\neq \varnothing$. Denote by $w(V')=\sum_{e\in \delta(V')}w_e$ the total weight of the edges covered by $V'$, {namely the coverage of $V'$. }   

Let $\mathcal M_l=(V,\mathcal I_l)$, $\mathcal M_f=(V,\mathcal I_f)$ be two matroids both on the ground vertex set $V$, and $\mathcal P_l,\mathcal P_f\subseteq [0,1]^n$ are their corresponding  matroid polytopes. In the RMVCI, we assume both players, the leader and the follower,  are subject to matroid constraints, that is,  for the leader, the feasible vertex subset he can choose to interdict 
must be an independent set of $\mathcal M_l$, and  the follower 
is limited to select an independent set of $\mathcal M_f$.

In the RMVCI,  $\mathcal I_l\subseteq 2^V$ is the set of all feasible subsets of vertices for the leader to interdict, and a \emph{mixed} (or \emph{randomized}) strategy over $\mathcal I_l$ is given by a probability distribution $\pi:\mathcal I_l\to[0,1]$, where $\pi(S)$ is the probability of interdicting the vertex subset $S$. 
Denote the set of all mixed strategies over $\mathcal I_l$ by $\Delta(\mathcal I_l)$.  The leader’s goal is to minimize the follower’s {optimal} expected payoff, expressed as $\Theta(\pi)$, by choosing an optimal probability distribution $\pi\in \Delta(\mathcal I_l)$,
\begin{align}
    \min_{\pi} \quad & \Theta(\pi) \label{eq:inter}\\
    \text{s.t.} \quad & \pi\in \Delta(\mathcal I_l) \nonumber
\end{align}
Once the leader commits to a mixed strategy $\pi$, the follower observes $\pi$ and then selects an independent vertex set of matroid $\mathcal M_f$ to maximize its successful edge-coverage. 
Let $x\in\{0,1\}^n$ denote the follower’s  decision vector {of selected vertices}, where $x_i=1$ indicates that vertex $i\in V$ is selected. The follower then aims to maximize the expected total weight of edges {that are covered} by the selected and uninterdicted vertices, which can be written as 
{\begin{align*}
  \Theta(\pi)=\max_{x}\quad&\sum_{S\in\mathcal I_l}\pi(S)w(A(x)\setminus S)\nonumber\\
   \text{s.t.} \quad & x\in\mathcal P_f\cap\{0,1\}^n\nonumber
\end{align*}
or equivalently, }
\begin{align}\label{eq:secon}
  \Theta(\pi)
  = \max_{x}\quad & \underset{\mathbf T_{\pi}\subseteq A(x)}{\mathbb{E}}[w(\mathbf T_{\pi})]\!:=\!\!\sum_{T\subseteq A(x)}\!\!\!\!w(T)\sum_{\substack{
S \in \mathcal{I}_l:\\
S \cap T = \varnothing,\ A(x)\setminus T \subseteq S
}}
\pi(S)\\
   \text{s.t.} \quad & x\in\mathcal P_f\cap\{0,1\}^n\nonumber
\end{align}
where $\mathbf{T}_{\pi}$ represents the random subset of $A(x)$ that 
are not interdicted by the leader with respect to the random strategy $\pi$ (i.e., the set of truly effective vertices),   
and the constraint ensures that  $x$ corresponds to an independent set of the follower’s matroid $\mathcal M_f$. 
{The restriction $S\cap T=\varnothing$ means that no vertex in $T$ is interdicted, and $A(x)\setminus T\subseteq S$ means that all vertices in $A(x)\setminus T$ are interdicted. }

\section{An Approximation Framework for Bilevel Optimization}\label{framework}

{We begin by describing a general framework for designing approximation algorithms for bilevel optimization problems.
The central idea is to replace the follower's value function by a surrogate objective that remains provably close to the original one while being easier to optimize at the leader level. This perspective is particularly useful when the follower's optimization problem is computationally difficult, but it is not restricted to such settings. The RMVCI problem studied in this paper provides a concrete illustration of the framework.}

Consider a bilevel min--max problem of the form
\begin{equation}\label{eq:gene}
\min_{y\in \mathcal Q_1}
\max_{x\in \mathcal Q_2(y)}
f(y,x),
\end{equation}
where $\mathcal Q_1$ and $\mathcal Q_2$ denote the feasible regions of the leader and follower, respectively.
When the follower's optimization problem
\[
\max_{x\in\mathcal Q_2(y)} f(y,x)
\]
can be solved efficiently, a variety of exact approaches become available, including dualization and value-function reformulations. In many applications, however, the follower's problem is itself NP-hard, rendering these classical approaches ineffective.

{We therefore turn to approximation algorithms. A feasible leader solution \(\hat y\in\mathcal Q_1\) is called an \(\alpha\)-approximate solution of \eqref{eq:gene}, for some \(\alpha\ge 1\), if
\[
\max_{x\in \mathcal Q_2(\hat y)}
f(\hat y,x)\le \alpha\cdot \mathrm{OPT},
\]
where $\mathrm{OPT}$ is the optimal value of \eqref{eq:gene}. 
An algorithm is said to have approximation ratio \(\alpha\), or to be an \(\alpha\)-approximation algorithm, if it returns an \(\alpha\)-approximate solution in polynomial time for every instance.}

The purpose of this section is to formalize a simple principle for obtaining approximation algorithms in such settings. Rather than attempting to solve the follower's problem exactly, we seek a surrogate value function that approximates the follower's optimal response while leading to a more tractable outer optimization problem.

Suppose that there exists a surrogate function
$\tilde f(y)$ satisfying
\begin{equation}\label{surrogate}
\frac1{c_2}\tilde f(y)
\le
\max_{x\in\mathcal Q_2(y)}f(y,x)
\le
c_1\tilde f(y)
\qquad
\forall y\in\mathcal Q_1,
\end{equation}
for some constants $c_1,c_2\ge 1$.
Intuitively, $\tilde f(y)$ provides a multiplicative approximation of the follower's optimal value function. We may therefore consider the surrogate optimization problem
$
\min_{y\in\mathcal Q_1}\tilde f(y),
$
which replaces the original bilevel problem by a potentially simpler optimization problem.

\begin{proposition}\label{lem:high}
Suppose that $\tilde f(y)$ is a surrogate function satisfying inequalities \eqref{surrogate}.
If the surrogate problem
\[
\min_{y\in\mathcal Q_1}\tilde f(y)
\]
admits a polynomial-time $c_3$-approximation algorithm ($c_3\ge 1$),
then the original bilevel problem~\eqref{eq:gene}
admits a polynomial-time
$c_1c_2c_3$-approximation algorithm.
\end{proposition}
\begin{proof}
 Let $\tilde y$ and $y'$ denote the optimal and $c_3$-approximation solutions to the relaxed problem $\min_{y\in \mathcal{Q}_1}\tilde f(y)$, respectively. Then, by definition of approximation, $$\tilde f(y')\le c_3\cdot \tilde f(\tilde y).$$ 
 
 Let $y^*$ be an optimal solution to the original bilevel problem  $\min_{y\in \mathcal Q_1}\max_{x\in\mathcal Q_2(y)}f(y,x)$. Using the bounding inequalities, we have  
  $$\max_{x\in\mathcal Q_2(y')}f(y',x)\le c_1\cdot \tilde f(y')\le c_1c_3\cdot \tilde f(\tilde y)\le c_1c_3\cdot  \tilde f(y^*)\le c_1c_2c_3\cdot \max_{x\in\mathcal Q_2(y^*)}f(y^*,x).$$
  The third inequality follows from the optimality of  $\tilde y$ for the problem $\min_{y\in \mathcal{Q}_1}\tilde f(y)$. 
 Hence, $y'$ yields a $c_1c_2c_3$-approximation for the bilevel problem~\eqref{eq:gene}.
\end{proof}

Proposition~\ref{lem:high} separates the design of an approximation algorithm into two conceptually distinct tasks:
\begin{enumerate}
\item Construct a surrogate value function that approximates the follower's optimal response.

\item Develop an approximation algorithm for the resulting surrogate optimization problem.
\end{enumerate}
The overall approximation guarantee is obtained by combining the approximation losses incurred in these two steps.

\paragraph{Sources of surrogate functions.}
The framework is deliberately agnostic to how the surrogate function is constructed.
In some settings, a surrogate may arise from a mathematical relaxation of the follower's problem, such as a linear or semidefinite programming relaxation. In other settings, it may be induced by an approximation algorithm for the follower. More generally, the surrogate may be obtained by modifying the objective function, relaxing certain dependencies, aggregating variables, or exploiting structural properties of the underlying problem.

The framework only requires the existence of a function $\tilde f(y)$ satisfying suitable approximation bounds. How such a function is constructed is problem-dependent and constitutes the main algorithmic challenge.
A high-quality surrogate should satisfy two competing requirements.
On the one hand, it should closely approximate the follower's optimal value function.
On the other hand, it should preserve enough structure to make the resulting outer optimization problem computationally tractable.

These two goals are often in conflict.
A very accurate surrogate may lead to an optimization problem that is no easier than the original bilevel model, whereas an overly simplified surrogate may yield poor approximation guarantees.
The effectiveness of the framework therefore depends critically on balancing approximation quality and algorithmic tractability.

The RMVCI problem studied in this paper illustrates two distinct uses of the framework.
We first construct a surrogate based on the LP relaxation of the follower's optimization problem. This surrogate provides a tight approximation of the follower's value function, but the resulting bilevel optimization problem remains computationally difficult.

We therefore develop a second surrogate by structurally modifying the follower's objective and replacing joint interdiction probabilities with marginal probabilities. Although this surrogate yields a weaker approximation guarantee, it possesses a substantially more tractable optimization structure. The final $2$-approximation algorithm is obtained from this second surrogate.

\section{The follower's problem}\label{sec:inner}
This section studies the follower's optimization problem under a fixed randomized interdiction strategy of the leader.
Our analysis has two purposes.
First, we investigate the approximability of the follower's problem itself by deriving an integer programming formulation, analyzing its linear relaxation, and designing a rounding algorithm.
Second, the relaxation developed in this section provides a natural candidate surrogate objective for the bilevel approximation framework introduced in Section~\ref{framework}.

The latter role is important for the leader's problem studied in Section~\ref{sec:out}.
Indeed, the LP relaxation obtained below approximates the follower's true value function within a tight factor of $4/3$.
From the perspective of approximation quality, this makes it a highly attractive surrogate.
As will be discussed in Section~\ref{sec:out}, however, approximation quality alone does not guarantee that the resulting bilevel optimization problem is tractable.

\subsection{An integer programming formulation}

Fix a randomized interdiction strategy $\pi\in \Delta(\mathcal I_l)$ of the leader.
The follower observes the induced protection distribution and chooses an independent set of vertices in the matroid $\mathcal M_f$.
Let $x\in \mathcal P_f\cap\{0,1\}^n$ denote the follower's decision vector, and let $A(x)=\{v\in V:x_v=1\}$ be the corresponding set of attacked vertices.

Recall from~\eqref{eq:secon} that the follower's optimal value is
\begin{align*}
\Theta(\pi)&=\max_{x\in\mathcal P_f\cap\{0,1\}^n}\quad\underset{\mathbf T_{\pi}}{\mathbb{E}}[w(\mathbf T_{\pi})]\\
        &=\max_{x\in\mathcal P_f\cap\{0,1\}^n} \sum_{T\subseteq A(x)}\text{Pr}(\mathbf T_{\pi}=T)\sum_{e\in\delta(T)}w_e\\
        &=\max_{x\in\mathcal P_f\cap\{0,1\}^n}\sum_{e\in\delta(A(x))}w_e\cdot\text{Pr}(|\mathbf T_{\pi}\cap e|\neq 0)\\
        &=\max_{x\in\mathcal P_f\cap\{0,1\}^n} \big(\sum_{\substack{e=(u,v)\in E:\\u\in A(x), v\notin A(x)}}\!\!\!\!\!w_e\text{Pr}(u\!\in\! \mathbf T_{\pi})+\sum_{\substack{e=(u,v)\in E:\\u,v\in A(x)}}\!\!\!\!\!w_e\text{Pr}(u\!\in\! \mathbf T_{\pi})\vee v\!\in \!\mathbf T_{\pi})\big)\\
        &=\max_{x\in\mathcal P_f\cap\{0,1\}^n} \big(\sum_{\substack{e=(u,v)\in E:\\u\in A(x), v\notin A(x)}}\!\!\!\!\!w_e(1-\sum_{S
    :u\in S}\pi(S))+\sum_{\substack{e=(u,v)\in E:\\u,v\in A(x)}}\!\!\!\!\!w_e(1-\sum_{S:u,v\in S}\pi(S))\big).
\end{align*}
where $\mathbf T_\pi$ denotes the random subset of attacked vertices that are not interdicted by the leader.

To expose the combinatorial structure of this problem, we rewrite the follower's objective in terms of the marginal and pairwise interdiction probabilities induced by $\pi$.
For each vertex $v\in V$, define
\[
q_v(\pi):=\sum_{S:v\in S}\pi(S),
\]
the marginal probability that vertex $v$ is interdicted.
For each edge $e=(u,v)\in E$, define
\[
q_{uv}(\pi):=\sum_{S:u,v\in S}\pi(S),
\]
the probability that both endpoints of $e$ are simultaneously interdicted.
Since $\pi$ is fixed throughout this section, we often write $q_v$ and $q_{uv}$ when no confusion can arise.

For each edge $e=(u,v)\in E$, define
\[
w_e^u:=w_e(1-q_u),\qquad
w_e^v:=w_e(1-q_v),\qquad
w_e^{uv}:=w_e(1-q_{uv}).
\]
Here $w_e^u$ and $w_e^v$ represent the expected contributions of edge $e$ when it is successfully covered through endpoint $u$ or $v$, respectively, while $w_e^{uv}$ represents the expected contribution when both endpoints are selected by the follower.

With this notation, the follower's problem can be formulated as the following integer linear program:
\begin{align}
\max\quad
&\sum_{e=(u,v)\in E}
\left[
x_u w_e^u+x_v w_e^v
+z_e\bigl(w_e^{uv}-w_e^u-w_e^v\bigr)
\right]
\label{eq:ilp}\\
\text{s.t.}\quad
&x\in \mathcal P_f, \nonumber\\
&z_e\ge x_u+x_v-1, \ \ \ \ \forall e=(u,v)\in E, \nonumber\\
&z_e\le x_u,\quad z_e\le x_v, \ \ \ \  \forall e=(u,v)\in E, \nonumber\\
&x_v\in\{0,1\},\quad z_e\in\{0,1\}, \ \  \forall v\in V,\ \forall e\in E. \nonumber
\end{align}

\begin{lemma}\label{lem:ilp-equivalence}
The integer program~\eqref{eq:ilp} is equivalent to the follower's problem~\eqref{eq:secon}.
\end{lemma}

\begin{proof}
In~\eqref{eq:ilp}, the variable $x_v$ indicates whether vertex $v$ is selected by the follower, while $z_e$ indicates whether both endpoints of edge $e$ are selected.
For an integral solution $x$, the constraints imply that $z_e=1$ if and only if $x_u=x_v=1$ for $e=(u,v)$.

Consider an edge $e=(u,v)$.
If only $u$ is selected, then the edge contributes $w_e^u$ in expectation.
If only $v$ is selected, it contributes $w_e^v$.
If both endpoints are selected, the expected contribution is $w_e^{uv}$, not $w_e^u+w_e^v$, because edge $e$ should be counted only once.
The correction term
\[
z_e\bigl(w_e^{uv}-w_e^u-w_e^v\bigr)
\]
therefore ensures that the total contribution of $e$ is exactly $w_e^{uv}$ when both endpoints are selected.
Summing over all edges gives precisely
\[
\mathbb E_{\mathbf T_\pi}\bigl[w(\mathbf T_\pi)\bigr].
\]
Thus~\eqref{eq:ilp} has the same feasible follower decisions and the same objective value as~\eqref{eq:secon}.
\end{proof}

The following elementary property will be used repeatedly.
It shows that the correction coefficient associated with $z_e$ is always non-positive.

\begin{lemma}\label{lem:negative-coeff}
For every edge $e=(u,v)\in E$,
\[
w_e^{uv}-w_e^u-w_e^v\le 0.
\]
\end{lemma}

\begin{proof}
By definition,
\[
w_e^{uv}-w_e^u-w_e^v
=
w_e(1-q_{uv})-w_e(1-q_u)-w_e(1-q_v)
=
w_e(q_u+q_v-q_{uv}-1).
\]
Using inclusion--exclusion,
\[
q_u+q_v-q_{uv}
=
\sum_{S:u\in S\ \mathrm{or}\ v\in S}\pi(S)
\le 1.
\]
Hence $w_e^{uv}-w_e^u-w_e^v\le 0$.
\end{proof}

Because the coefficient of $z_e$ is non-positive, an optimal solution always chooses $z_e$ as small as possible.
Therefore the constraints $z_e\le x_u$ and $z_e\le x_v$ are redundant for the LP relaxation, and the relaxation can be written as
\begin{align}
\max\quad
&\sum_{e=(u,v)\in E}
\left[
x_u w_e^u+x_v w_e^v
+z_e\bigl(w_e^{uv}-w_e^u-w_e^v\bigr)
\right]
\label{eq:lplp}\\
\text{s.t.}\quad
&x\in \mathcal P_f, \nonumber\\
&z_e\ge x_u+x_v-1, \ \ \forall e=(u,v)\in E, \nonumber\\
&0\le z_e\le 1, \ \ \ \forall e\in E. \nonumber
\end{align}
Let $Z^{ILP}_{\mathrm{inner}}(\pi)$ and $Z^{LP}_{\mathrm{inner}}(\pi)$ denote the optimal values of~\eqref{eq:ilp} and~\eqref{eq:lplp}, respectively.
Clearly,
\[
Z^{ILP}_{\mathrm{inner}}(\pi)\le Z^{LP}_{\mathrm{inner}}(\pi).
\]

\subsection{The LP relaxation and a smoothing function}

The LP relaxation admits an equivalent formulation involving only the variables $x$.
Indeed, since the coefficient of $z_e$ is non-positive, an optimal solution of~\eqref{eq:lplp} satisfies
\[
z_e=\max\{0,x_u+x_v-1\}
\qquad \forall e=(u,v)\in E.
\]
Substituting these values into the objective yields
\begin{equation}\label{eq:lpx}
Z^{LP}_{\mathrm{inner}}(\pi)
=
\max_{x\in\mathcal P_f} L^\pi(x),
\end{equation}
where
\begin{equation}\label{eq:Lpi}
L^\pi(x)
:=
\sum_{e=(u,v)\in E}
\left[
x_u w_e^u+x_v w_e^v
+
\max\{0,x_u+x_v-1\}
\bigl(w_e^{uv}-w_e^u-w_e^v\bigr)
\right].
\end{equation}

The function $L^\pi(x)$ is piecewise linear and contains the non-differentiable terms
$\max\{0,x_u+x_v-1\}$.
To facilitate both the rounding procedure and the integrality-gap analysis, we introduce a smoother auxiliary function by replacing
\[
\max\{0,x_u+x_v-1\}
\]
with the product $x_ux_v$.
Define
\begin{equation}\label{eq:Fpi}
F^\pi(x)
:=
\sum_{e=(u,v)\in E}
\left[
x_u w_e^u+x_v w_e^v
+
x_ux_v
\bigl(w_e^{uv}-w_e^u-w_e^v\bigr)
\right].
\end{equation}
When $\pi$ is fixed, we write $L(x)$ and $F(x)$ for simplicity.

The next lemma establishes the key relationship between $L$ and $F$.
It is the central technical ingredient of this section.

\begin{lemma}\label{lem:LF}
For every $x\in \mathcal P_f$,
\[
F(x)\ge \frac34 L(x).
\]
Moreover, if $x\in\mathcal P_f\cap\{0,1\}^n$, then
\[
F(x)=L(x).
\]
\end{lemma}

\begin{proof}
Fix an edge $e=(u,v)\in E$.
It suffices to compare the contribution of $e$ to $F(x)$ and $L(x)$.
For notational simplicity, write
\[
a:=1-q_u,\qquad b:=1-q_v,\qquad c:=1-q_{uv}.
\]
Then
\[
w_e^u=w_ea,\qquad w_e^v=w_eb,\qquad w_e^{uv}=w_ec.
\]
By Lemma~\ref{lem:negative-coeff}, we have
\(
a+b-c\ge 0.
\)
We consider two cases.

First suppose that $x_u+x_v\le 1$.
Then the contribution of $e$ to $L(x)$ is
\[
w_e(x_ua+x_vb),
\]
while its contribution to $F(x)$ is
\[
w_e\left(x_ua+x_vb-x_ux_v(a+b-c)\right).
\]
Since
\[
a+b-c=1-q_u-q_v+q_{uv}\le \min\{1-q_u,1-q_v\}=\min\{a,b\},
\]
and
\[
x_ux_v\le \frac{(x_u+x_v)^2}{4}\le \frac{x_u+x_v}{4},
\]
we obtain
\[
x_ux_v(a+b-c)
\le
\frac14(x_ua+x_vb).
\]
Therefore the contribution of $e$ to $F(x)$ is at least $3/4$ of its contribution to $L(x)$.

Now suppose that $x_u+x_v\ge 1$.
The difference between the contribution of $e$ to $F(x)$ and $3/4$ times its contribution to $L(x)$ equals
\[
w_e\left[
\frac14x_ua+\frac14x_vb
+
\left(x_ux_v-\frac34(x_u+x_v-1)\right)(c-a-b)
\right].
\]
Equivalently, using $c-a-b=-(a+b-c)$, this expression can be written as
\[
w_e\left[
(a+b-c)\left(x_u+x_v-x_ux_v-\frac34\right)
+
\frac14 x_u(q_v-q_{uv})
+
\frac14 x_v(q_u-q_{uv})
\right].
\]
Each term is non-negative.
Indeed,
\[
a+b-c=1-q_u-q_v+q_{uv}\ge 0,
\]
\[
q_v-q_{uv}\ge 0,\qquad q_u-q_{uv}\ge 0,
\]
and
\[
x_u+x_v-x_ux_v\ge \frac34
\qquad \text{whenever } x_u+x_v\ge 1,\quad x_u,x_v\in[0,1].
\]
Thus the contribution of $e$ to $F(x)$ is again at least $3/4$ of its contribution to $L(x)$.

Summing over all edges proves $F(x)\ge \frac34 L(x)$ for every $x\in\mathcal P_f$.

Finally, if $x$ is integral, then for each edge $e=(u,v)$,
\[
x_ux_v=\max\{0,x_u+x_v-1\}.
\]
Hence $F(x)=L(x)$ for all integral $x$.
\end{proof}

Lemma~\ref{lem:LF} has two consequences.
First, $F$ agrees exactly with the follower's objective on integral solutions.
Second, over the entire matroid polytope, $F$ provides a uniform $3/4$ lower approximation of the LP objective $L$.
These two properties allow us to round an optimal fractional solution of the LP relaxation without losing more than a factor of $4/3$.

\subsection{Approximation algorithm and integrality gap}

We now use Lemma~\ref{lem:LF} to obtain an approximation algorithm for the follower's problem and, at the same time, an integrality-gap bound for the LP~\eqref{eq:lplp}.

\begin{theorem}\label{thm:gap}
For every fixed interdiction strategy $\pi$, one can compute in polynomial time an integral solution $\hat x\in \mathcal P_f\cap\{0,1\}^n$ such that
\[
Z^{ILP}_{\mathrm{inner}}(\pi)
\ge
L(\hat x)
=
F(\hat x)
\ge
\frac34 Z^{LP}_{\mathrm{inner}}(\pi).
\]
Consequently, the follower's problem admits a polynomial-time $4/3$-approximation algorithm.
\end{theorem}

\begin{proof}
Let
\[
x^*\in \arg\max_{x\in\mathcal P_f} L(x)
\]
be an optimal solution of the LP relaxation~\eqref{eq:lpx}.
Then
\[
L(x^*)=Z^{LP}_{\mathrm{inner}}(\pi).
\]
Since $\mathcal P_f$ is a matroid polytope and we assume access to an independence oracle, $\mathcal P_f$ admits a polynomial-time separation oracle.
By the equivalence of separation and optimization, $x^*$ can be computed in polynomial time.

It remains to round $x^*$ to an integral point $\hat x\in\mathcal P_f$ without decreasing the value of $F$.
We use pipage rounding (see \cite{ageev2004pipage,calinescu2011maximizing}). {First, for any point $x\in\mathcal P_f$ and any pair of indices $i,j$, 
 we show that function $F$ is \emph{convex in the direction} $e_i-e_j$, where $e_i=(0,\ldots,0,1,0,\ldots,0)$ is the $i$‑th standard basis vector. That is, 
 the one-dimensional function 
 $$g_{ij}(t):=F(x+t(e_i-e_j)), ~t\in\mathbb R$$
 is  convex. Recall the expression of $F$ in \eqref{eq:Fpi}.
 Since the sum of finitely many convex functions is convex, it suffices to show that every term of $g_{ij}(t)$ is convex. For each term with respect to edge $e=(u,v)\in E$, if $\{u,v\}\cap\{i,j\}=\varnothing$, the edge contribution is constant. If $|\{u,v\}\cap\{i,j\}|=1$, it is linear with $t$. If $\{u,v\}=\{i,j\}$, there is a quadratic sub-term $(w_e^u+w_e^v-w_e^{uv})t^2$, which is convex because the coefficient  $(w_e^u + w_e^v - w_e^{uv})$ is non-negative by Lemma~\ref{lem:negative-coeff}. Hence $g_{ij}$ is convex. }

 The standard pipage-rounding argument  now applies.
{Since \(\mathcal P_f\) is a matroid independent-set polytope, we first reduce
to a matroid base polytope in the standard way by adding dummy elements of
zero contribution. Equivalently, we extend the matroid so that every
independent set can be completed to a base, and define \(F\) to be independent
of the dummy coordinates. The convexity property along directions
\(e_i-e_j\) is preserved, since dummy elements have zero weight and introduce
no new nonzero edge terms. }

Applying pipage rounding on this base polytope, if the current point \(x\)
is not a vertex, there exist two coordinates \(i,j\) and positive values
\(\alpha,\beta\) such that both
\[
x+\alpha(e_i-e_j)
\quad\text{and}\quad
x-\beta(e_i-e_j)
\]
remain feasible and lie on faces of strictly smaller dimension. Moreover,
such \(i,j,\alpha,\beta\) can be found efficiently. By convexity of
\(g_{ij}\), at least one of these two moves does not decrease \(F\).
We perform such a non-decreasing move and repeat.
The process terminates after polynomially many iterations at a vertex of the
extended base polytope. {Projecting this vertex back to the original elements
gives an integral point \(\hat x\in\mathcal P_f\). Since \(F\) does not depend
on the dummy coordinates,} and every step was non-decreasing, we have
\[
F(\hat x)\ge F(x^*).
\]

Using Lemma~\ref{lem:LF}, we obtain
\[
L(\hat x)
=
F(\hat x)
\ge
F(x^*)
\ge
\frac34 L(x^*)
=
\frac34 Z^{LP}_{\mathrm{inner}}(\pi).
\]
Since $\hat x$ is an integral feasible solution, its value is at most the optimum integral value, and hence
\[
Z^{ILP}_{\mathrm{inner}}(\pi)
\ge L(\hat x)
=
F(\hat x)
\ge
\frac34 Z^{LP}_{\mathrm{inner}}(\pi).
\]
This proves the theorem.
\end{proof}

Theorem~\ref{thm:gap} implies that the LP relaxation~\eqref{eq:lplp} has integrality gap at most $4/3$.
This bound is asymptotically tight.
Consider the special case in which the follower's matroid is a rank-$k$ uniform matroid with $k=n/2$, the graph is the complete graph $K_n$, all edge weights are equal to one, and the leader performs no interdiction.
In this setting, the follower's problem reduces to the classical maximum vertex-coverage problem on $K_n$ under a cardinality constraint.
Any integral solution selecting $k$ vertices covers
$k(n-1)-\binom{k}{2}$ edges.
On the other hand, the fractional solution $x_v=\frac12$ for all $v\in V$ is feasible for~\eqref{eq:lplp} and attains value $\binom{n}{2}$.
Therefore, with $k=n/2$,
\[
\frac{Z^{LP}_{\mathrm{inner}}(\pi)}
     {Z^{ILP}_{\mathrm{inner}}(\pi)}
\ge
\frac{\binom n2}{k(n-1)-\binom k2}
\longrightarrow
\frac43
\qquad\text{as } n\to\infty.
\]

\begin{corollary}\label{cor:gap}
The integrality gap of LP~\eqref{eq:lplp} is exactly $4/3$.
\end{corollary}

Beyond its algorithmic implications for the follower's problem, the LP relaxation developed in this section plays an additional role in the bilevel setting.
For every interdiction strategy $\pi$, the value
$Z^{LP}_{\mathrm{inner}}(\pi)$
provides a tight $4/3$-approximation of the follower's true value function.
Consequently, within the approximation framework of Section~\ref{framework}, it constitutes a natural candidate surrogate objective for the leader's optimization problem.

\section{The leader's interdiction problem}\label{sec:out}

We now return to the leader's randomized interdiction problem~\eqref{eq:inter}.
The analysis in Section~\ref{sec:inner} provides a natural first candidate surrogate for the follower's value function.
For every fixed interdiction strategy $\pi$, we have shown that
\[
\frac34 Z^{LP}_{\mathrm{inner}}(\pi)
\le
Z^{ILP}_{\mathrm{inner}}(\pi)
\le
Z^{LP}_{\mathrm{inner}}(\pi).
\]
Thus, the LP relaxation gives a $4/3$-approximation of the follower's value function.
This suggests the LP-based surrogate problem
\begin{equation}\label{eq:minlx}
\min_{\pi\in\Delta(\mathcal I_l)}
Z^{LP}_{\mathrm{inner}}(\pi)
=
\min_{\pi\in\Delta(\mathcal I_l)}
\max_{x\in\mathcal P_f} L^\pi(x).
\end{equation}
By Proposition~\ref{lem:high}, any $\alpha$-approximate solution of~\eqref{eq:minlx} would yield a
$\frac{4\alpha}{3}$-approximation algorithm for the original interdiction problem.
Hence, from the viewpoint of approximation quality,~\eqref{eq:minlx} is a highly attractive surrogate model.

However, as discussed {in Section~\ref{sec:limit}}, this LP-based surrogate does not directly lead to an efficient algorithm for the leader's problem under general matroid constraints.
This motivates a second surrogate objective {in Section~\ref{subsec:mar}}, based only on marginal interdiction probabilities {on vertices}.
Although this second surrogate is less tight than the LP relaxation, it has a much more tractable optimization structure and ultimately yields a polynomial-time $2$-approximation algorithm.

\subsection{The LP-based surrogate and its limitation}\label{sec:limit}

A standard approach for bilevel optimization problems with tractable inner linear programs is the dualize-and-combine technique.
Since the inner maximization problem in~\eqref{eq:minlx} is the LP relaxation~\eqref{eq:lplp}, one may dualize the follower's LP and combine the resulting dual with the leader's outer minimization problem.

For a fixed interdiction strategy $\pi$, the dual of the follower's LP relaxation can be written using dual variables $\alpha_S$ for the matroid-rank inequalities of $\mathcal P_f$, and variables $\gamma_e,\beta_e$ associated with the edge constraints.
Combining this dual with the leader's decision variables gives the following single-level formulation:
\begin{align}
\min\quad
&\sum_{S\subseteq V} r_f(S)\alpha_S+\sum_{e\in E}(\gamma_e+\beta_e)
\label{eq:mindual}\\
\text{s.t.}\quad
&\sum_{S:v\in S}\alpha_S+\sum_{e:v\in e}\gamma_e
\ge w_e^v(\pi),
&&\forall v\in V, \nonumber\\
&\beta_e-\gamma_e
\ge w_e^{uv}(\pi)-w_e^u(\pi)-w_e^v(\pi),
&&\forall e=(u,v)\in E, \nonumber\\
&\sum_{S\in\mathcal I_l}\pi(S)=1, \nonumber\\
&\alpha_S\ge 0,\quad \gamma_e\ge 0,\quad \beta_e\ge 0,\quad \pi(S)\ge 0,
&&\forall S\subseteq V,\ \forall e\in E. \nonumber
\end{align}
Here, {$r_f(\cdot)$ is the rank function of matroid $\mathcal M_f$, and}
\[
w_e^u(\pi)=w_e(1-q_u(\pi)),\qquad
w_e^v(\pi)=w_e(1-q_v(\pi)),
~\text{~and~}~
w_e^{uv}(\pi)=w_e(1-q_{uv}(\pi)),
\]
where
\[
q_u(\pi)=\sum_{S:u\in S}\pi(S),
\qquad
q_{uv}(\pi)=\sum_{S:u,v\in S}\pi(S).
\]

The formulation~\eqref{eq:mindual} is an exact reformulation of the LP-based surrogate problem~\eqref{eq:minlx}.
Consequently, a polynomial-time exact algorithm for~\eqref{eq:mindual} would imply a polynomial-time $4/3$-approximation algorithm for the original interdiction problem.


{The main difficulty is that~\eqref{eq:mindual} does not admit an apparent polynomial-size representation for general matroid constraints. 
Indeed, an exact formulation involves a distribution over the feasible interdiction sets of the leader, while the follower-side dual contains variables indexed by the rank inequalities of the follower matroid polytope. 
More importantly, the constraints depend not only on the marginal interdiction probabilities \(q_v(\pi)\), but also on the pairwise probabilities \(q_{uv}(\pi)\). 
Thus one must optimize over the pairwise marginal polytope induced by the leader matroid, which is substantially more complicated.} 
Consequently, although the LP relaxation provides an excellent approximation of the follower's value for a fixed leader distribution $\pi$, the resulting bilevel surrogate problem remains difficult to solve in general.


This observation highlights an important point in applying the approximation framework of Section~\ref{framework}: a useful surrogate must be judged not only by how accurately it approximates the follower's value function, but also by whether it leads to a tractable leader-side optimization problem.
We therefore introduce a second surrogate objective, obtained by simplifying the dependence on pairwise interdiction probabilities.

\subsection{A structural surrogate based on marginal probabilities}\label{subsec:mar}

Recall that for an integral follower decision $x\in\mathcal P_f\cap\{0,1\}^n$,
\[
L^\pi(x)
=
\sum_{e=(u,v)\in E}
\left[
x_u w_e^u(\pi)
+
x_v w_e^v(\pi)
+
\max\{0,x_u+x_v-1\}
\bigl(
w_e^{uv}(\pi)-w_e^u(\pi)-w_e^v(\pi)
\bigr)
\right].
\]
The main obstacle is the term
\[
w_e^{uv}(\pi)=w_e(1-q_{uv}(\pi)),
\]
which depends on the probability that both endpoints of an edge are interdicted simultaneously.
This pairwise dependence is difficult to handle in the leader's optimization problem.

We therefore replace the pairwise quantity $1-q_{uv}(\pi)$ by the simpler marginal expression
\[
(1-q_u(\pi))+(1-q_v(\pi)).
\]
Equivalently, define
\[
\widetilde w_e^{uv}(\pi)
:=
w_e\bigl(2-q_u(\pi)-q_v(\pi)\bigr)
=
w_e^u(\pi)+w_e^v(\pi).
\]
Substituting $\widetilde w_e^{uv}(\pi)$ for $w_e^{uv}(\pi)$ in the objective $L^{\pi}(x)$ cancels the correction term and yields the surrogate function
\begin{equation}\label{eq:tildeL}
\widetilde L^{q(\pi)}(x)
:=
\sum_{e=(u,v)\in E}
\left[
x_u w_e(1-q_u(\pi))
+
x_v w_e(1-q_v(\pi))
\right].
\end{equation}
The crucial feature of $\widetilde L^{q(\pi)}(x)$ is that it depends on $\pi$ only through the marginal vector
\[
q(\pi)=(q_v(\pi))_{v\in V}.
\]
In particular, the pairwise probabilities $q_{uv}(\pi)$ no longer appear.

The following lemma shows that this structural simplification loses at most a factor of two.

\begin{lemma}\label{lem:surrogate-bound}
For every randomized strategy $\pi$ and every integral follower solution
$x\in\mathcal P_f\cap\{0,1\}^n$,
\[
\frac12 \widetilde L^{q(\pi)}(x)
\le
L^\pi(x)
\le
\widetilde L^{q(\pi)}(x).
\]
\end{lemma}

\begin{proof}
Fix an edge $e=(u,v)\in E$.
We compare the contribution of $e$ to $L^\pi(x)$ and to $\widetilde L^{q(\pi)}(x)$.
If $x_u+x_v\le 1$, then at most one endpoint of $e$ is selected.
In this case, the two objectives have the same edge contribution
\(
x_u w_e^u(\pi)+x_v w_e^v(\pi),
\)
and the desired inequalities hold immediately.

Now suppose $x_u=x_v=1$.
The contribution of $e$ to $L^\pi(x)$ is
\[
w_e^{uv}(\pi)=w_e(1-q_{uv}(\pi)),
\]
whereas the contribution of $e$ to $\widetilde L^{q(\pi)}(x)$ is
\[
w_e^u(\pi)+w_e^v(\pi)
=
w_e(2-q_u(\pi)-q_v(\pi)).
\]
The upper bound follows from Lemma~\ref{lem:negative-coeff}, which implies
\[
w_e^{uv}(\pi)\le w_e^u(\pi)+w_e^v(\pi).
\]
For the lower bound, it suffices to show
\[
\frac{2-q_u(\pi)-q_v(\pi)}{2}
\le
1-q_{uv}(\pi).
\]
Since
\[
q_{uv}(\pi)\le \min\{q_u(\pi),q_v(\pi)\}
\le
\frac{q_u(\pi)+q_v(\pi)}{2},
\]
the inequality follows.
Summing over all edges proves the lemma.
\end{proof}

Lemma~\ref{lem:surrogate-bound} shows that
\[
\max_{x\in\mathcal P_f\cap\{0,1\}^n}
\widetilde L^{q(\pi)}(x)
\]
is a valid surrogate for the follower's true value function with constants $c_1=1$ and $c_2=2$ in Proposition~\ref{lem:high}.
We therefore consider the surrogate interdiction problem
\begin{equation}\label{eq:surrogate-pi}
\min_{\pi\in\Delta(\mathcal I_l)}
\max_{x\in\mathcal P_f\cap\{0,1\}^n}
\widetilde L^{q(\pi)}(x).
\end{equation}

\subsection{Marginal reformulation}

The advantage of~\eqref{eq:surrogate-pi} is that its objective depends on the leader's mixed strategy only through the marginal vector $q(\pi)$.
Every distribution $\pi\in\Delta(\mathcal I_l)$ induces a vector $q(\pi)\in\mathcal P_l$.
Conversely, every point $q\in\mathcal P_l$ can be represented as a convex combination of incidence vectors of independent sets of the leader's matroid.
Thus, the distributional problem~\eqref{eq:surrogate-pi} can be reformulated over the leader's matroid polytope.

\begin{lemma}\label{lem:marginal-equivalence}
The surrogate problem~\eqref{eq:surrogate-pi} has the same optimal value as
\begin{equation}\label{eq:surrogate-q}
\min_{q\in\mathcal P_l}
\max_{x\in\mathcal P_f\cap\{0,1\}^n}
\widetilde L^q(x),
\end{equation}
where
\begin{equation}\label{eq:tildeLq}
\widetilde L^q(x)
=
\sum_{e=(u,v)\in E}
\left[
x_u w_e(1-q_u)
+
x_v w_e(1-q_v)
\right].
\end{equation}
Moreover, an optimal marginal vector $q$ can be converted in polynomial time into a distribution over at most $n+1$ independent sets of $\mathcal M_l$.
\end{lemma}

\begin{proof}
Any distribution $\pi\in\Delta(\mathcal I_l)$ induces a marginal vector $q(\pi)\in\mathcal P_l$.
Since $\widetilde L^{q(\pi)}(x)$ depends only on these marginals, the value of~\eqref{eq:surrogate-q} is no larger than that of~\eqref{eq:surrogate-pi}.

Conversely, take any $q\in\mathcal P_l$.
By definition, $\mathcal P_l$ is the convex hull of incidence vectors of independent sets of $\mathcal M_l$.
Therefore, $q$ admits a representation
\[
q=\sum_{S\in\mathcal I_l}\pi(S)\chi^S
\]
for some distribution $\pi\in\Delta(\mathcal I_l)$, {where $\chi^S$ is the incidence vector of $S$.}
For such a distribution, the induced marginals are exactly $q$, and hence the objective value in~\eqref{eq:surrogate-pi} equals that in~\eqref{eq:surrogate-q}.


{Moreover, given \(q\), such a distribution \(\pi\) can be recovered in polynomial time
using the convex decomposition algorithm of Gr{\"o}tschel, Lov{\'a}sz, and
Schrijver (see Theorem~6.5.11 in~\cite{grotschel1988geometric}). Since
\(\mathcal P_l\subseteq\mathbb R^n\) is a matroid polytope and admits a polynomial-time separation oracle, this algorithm returns
vertices \( \chi^{S_1},\ldots,\chi^{S_k}\) of \(\mathcal P_l\) and coefficients
\(\lambda_1,\ldots,\lambda_k\ge 0\), with \(k\le n+1\) and
\(\sum_{i=1}^k\lambda_i=1\), such that
\(
q=\sum_{i=1}^k \lambda_i \chi^{S_i}.
\)
Because the vertices of a matroid  polytope are precisely the
incidence vectors of independent sets, this gives a distribution supported on
\(S_1,\ldots,S_k\).}
\end{proof}

We next show that~\eqref{eq:surrogate-q} can be solved in polynomial time.

\begin{lemma}\label{lem:solve-surrogate}
The marginal surrogate problem~\eqref{eq:surrogate-q} can be solved in polynomial time.
\end{lemma}

\begin{proof}
For a fixed $q$, the function $\widetilde L^q(x)$ is linear in $x$ with non-negative coefficients.
Thus, the inner maximization
$\max_{x\in\mathcal P_f\cap\{0,1\}^n}\widetilde L^q(x)$ is simply the problem
of finding a maximum-weight independent set in the follower matroid, and can be
solved by the greedy algorithm.

Now consider the epigraph formulation
\begin{align}
\min\quad & t \label{eq:epi}\\
\text{s.t.}\quad
&q\in\mathcal P_l, \nonumber\\
&t\ge \widetilde L^q(x),
\quad\forall x\in\mathcal P_f\cap\{0,1\}^n. \nonumber
\end{align}
This is a linear program in the variables $(q,t)$ with exponentially many constraints.
 {We separate over~\eqref{eq:epi} as follows.}
Given any candidate point $(q,t)$, we first test whether $q\in\mathcal P_l$ using the separation oracle for the leader matroid polytope.
If $q\notin\mathcal P_l$, the oracle returns a separating hyperplane.
Otherwise, we compute
\[
\max_{x\in\mathcal P_f\cap\{0,1\}^n}\widetilde L^q(x)
\]
by the greedy algorithm over the follower matroid.
If the optimal value is larger than $t$, the corresponding integral maximizer $x$ gives a violated epigraph constraint of~\eqref{eq:epi}. If the optimum value is at most $t$, then no epigraph constraint is violated.

Thus~\eqref{eq:epi} admits a polynomial-time separation oracle. By the
equivalence of separation and optimization, the epigraph formulation, and hence
the marginal surrogate problem~\eqref{eq:surrogate-q}, can be solved in
polynomial time.
\end{proof}

\subsection{Applying the framework}

We now combine the surrogate bounds with the marginal reformulation to obtain a $2$-approximation algorithm for the original leader's interdiction problem.

\begin{theorem}\label{thm:leader}
The RMVCI problem under matroid constraints admits a polynomial-time
$2$-approximation algorithm.
Moreover, the returned randomized interdiction strategy can be chosen
to have support size at most $n+1$.
\end{theorem}

\begin{proof}
By Lemma~\ref{lem:surrogate-bound},
the surrogate objective satisfies
\[
\frac12 
\max_{x\in\mathcal P_f\cap\{0,1\}^n}\widetilde L^{q(\pi)}(x)
\le
\max_{x\in\mathcal P_f\cap\{0,1\}^n}L^\pi(x)
\le
\max_{x\in\mathcal P_f\cap\{0,1\}^n}\widetilde L^{q(\pi)}(x).
\]
for every interdiction strategy $\pi$.
Hence the framework parameters are $c_1=1, c_2=2$.

By Lemma~\ref{lem:marginal-equivalence}
and Lemma~\ref{lem:solve-surrogate},
the surrogate  problem $\min_{\pi\in\Delta(\mathcal I_l)}
\max_{x\in\mathcal P_f\cap\{0,1\}^n}\widetilde L^{q(\pi)}(x)$ can be solved exactly in polynomial time.
Therefore $c_3=1$.

Applying Proposition~\ref{lem:high}
yields a polynomial-time
$c_1c_2c_3=2$
approximation algorithm.
The support-size bound follows from
Lemma~\ref{lem:marginal-equivalence}.
\end{proof}

\section{Discussion and Future Directions}\label{sec:con}

This paper studies a randomized network interdiction problem in which the follower's optimization problem is NP-hard. Unlike classical interdiction models, where tractable follower problems permit exact reformulations through dualization or value-function techniques, the RMVCI problem requires approximation methods at both levels of the bilevel program.

Our main contribution is not only a constant-factor approximation algorithm for RMVCI, but also a general approximation methodology for handling bilevel optimization problems. 
The key idea is to replace the follower's value function by a surrogate objective that simultaneously preserves approximation guarantees and exposes additional structure in the leader's optimization problem. In the RMVCI setting, this leads to a reformulation based on marginal interdiction probabilities and matroid polytopes, which ultimately enables polynomial-time optimization.

Several directions deserve further investigation. First, the approximation framework developed here is not restricted to vertex-cover objectives and may be applicable to broader classes of interdiction problems involving NP-hard packing, covering, or domination problems. Second, it would be interesting to identify conditions under which surrogate objectives can be constructed systematically rather than on a problem-specific basis. Finally, closing the gap between the current 2-approximation guarantee and the best achievable approximation ratio for RMVCI remains an important open question.

More broadly, we believe that approximation-oriented reformulations of follower value functions may provide a useful paradigm for bilevel optimization beyond network interdiction, particularly in settings where exact lower-level optimization is computationally prohibitive.

\paragraph{Funding Declaration.}
Chenhao Wang is supported by BNBU under grant UICR0400004-24B. The content is solely the responsibility of the authors and does not necessarily represent the official views of the funding agencies.

\paragraph{Data Availability Statement} This manuscript has no associated data.

\bibliographystyle{plain}
\bibliography{reference}

@book{grotschel1988geometric,
  title={Geometric algorithms and combinatorial optimization},
  author={Gr{\"o}tschel, Martin and Lov{\'a}sz, L{\'a}szl{\'o} and Schrijver, Alexander},
  year={1988},
volume={2},
  publisher={Springer}
}

@article{lozano2017value,
  title={A value-function-based exact approach for the bilevel mixed-integer programming problem},
  author={Lozano, Leonardo and Smith, J Cole},
  journal={Operations Research},
  volume={65},
  number={3},
  pages={768--786},
  year={2017},
  publisher={INFORMS}
}

@article{DBLP:journals/tcs/ChenZ13,
  author       = {Lin Chen and
                  Guochuan Zhang},
  title        = {Approximation algorithms for a bi-level knapsack problem},
  journal      = {Theoretical Computer Science},
  volume       = {497},
  pages        = {1--12},
  year         = {2013}
}

@inproceedings{chen2022approximation,
  title={Approximation Algorithms for Interdiction Problem with Packing Constraints},
  author={Chen, Lin and Wu, Xiaoyu and Zhang, Guochuan},
  booktitle={49th International Colloquium on Automata, Languages, and Programming (ICALP)},
  volume={229},
  pages={39},
  year={2022}
}

@article{calinescu2011maximizing,
  title={Maximizing a monotone submodular function subject to a matroid constraint},
  author={Calinescu, Gruia and Chekuri, Chandra and Pal, Martin and Vondr{\'a}k, Jan},
  journal={SIAM Journal on Computing},
  volume={40},
  number={6},
  pages={1740--1766},
  year={2011},
  publisher={SIAM}
}

@article{DBLP:journals/orl/CarvalhoLM18,
  author       = {Margarida Carvalho and
                  Andrea Lodi and
                  Patrice Marcotte},
  title        = {A polynomial algorithm for a continuous bilevel knapsack problem},
  journal      = {Operations Research Letters},
  volume       = {46},
  number       = {2},
  pages        = {185--188},
  year         = {2018}
}

@article{DBLP:journals/mp/CroceS20,
  author       = {Federico Della Croce and
                  Rosario Scatamacchia},
  title        = {An exact approach for the bilevel knapsack problem with interdiction
                  constraints and extensions},
  journal      = {Mathematical Programming},
  volume       = {183},
  number       = {1},
  pages        = {249--281},
  year         = {2020}
}

@article{frederickson1999increasing,
  title={Increasing the weight of minimum spanning trees},
  author={Frederickson, Greg N and Solis-Oba, Roberto},
  journal={Journal of Algorithms},
  volume={33},
  number={2},
  pages={244--266},
  year={1999},
  publisher={Elsevier}
}

@article{bertsimas2016power,
  title={On the power of randomization in network interdiction},
  author={Bertsimas, Dimitris and Nasrabadi, Ebrahim and Orlin, James B},
  journal={Operations Research Letters},
  volume={44},
  number={1},
  pages={114--120},
  year={2016},
  publisher={Elsevier}
}

@article{eor/TayyebiMS23,
  author       = {Javad Tayyebi and
                  Ankan Mitra and
                  Jorge A. Sefair},
  title        = {The continuous maximum capacity path interdiction problem},
  journal      = {European Journal of Operations Research},
  volume       = {305},
  number       = {1},
  pages        = {38--52},
  year         = {2023}
}

@article{weninger2025interdiction,
  title={Interdiction of minimum spanning trees and other matroid bases},
  author={Weninger, Noah and Fukasawa, Ricardo},
  journal={Mathematical Programming},
  pages={1--43},
  year={2025},
  publisher={Springer}
}

@article{chestnut2017interdicting,
  title={Interdicting structured combinatorial optimization problems with $\{$0, 1$\}$-objectives},
  author={Chestnut, Stephen R and Zenklusen, Rico},
  journal={Mathematics of Operations Research},
  volume={42},
  number={1},
  pages={144--166},
  year={2017},
  publisher={INFORMS}
}

@article{wollmer1964removing,
  title={Removing arcs from a network},
  author={Wollmer, Richard},
  journal={Operations Research},
  volume={12},
  number={6},
  pages={934--940},
  year={1964},
  publisher={INFORMS}
}

@article{church2004identifying,
  title={Identifying critical infrastructure: the median and covering facility interdiction problems},
  author={Church, Richard L and Scaparra, Maria P and Middleton, Richard S},
  journal={Annals of the Association of American Geographers},
  volume={94},
  number={3},
  pages={491--502},
  year={2004},
  publisher={Taylor \& Francis}
}

@article{fulkerson1977maximizing,
  title={Maximizing the minimum source-sink path subject to a budget constraint},
  author={Fulkerson, Delbert Ray and Harding, Gary C},
  journal={Mathematical Programming},
  volume={13},
  number={1},
  pages={116--118},
  year={1977},
  publisher={Springer}
}

@article{zenklusen2010matching,
  title={Matching interdiction},
  author={Zenklusen, Rico},
  journal={Discrete Applied Mathematics},
  volume={158},
  number={15},
  pages={1676--1690},
  year={2010},
  publisher={Elsevier}
}

@inproceedings{dinitz2013packing,
  title={Packing interdiction and partial covering problems},
  author={Dinitz, Michael and Gupta, Anupam},
  booktitle={International Conference on Integer Programming and Combinatorial Optimization (IPCO)},
  pages={157--168},
  year={2013},
  organization={Springer}
}

@article{ior/HolzmannS21,
  author       = {Tim Holzmann and
                  J. Cole Smith},
  title        = {The Shortest Path Interdiction Problem with Randomized Interdiction
                  Strategies: Complexity and Algorithms},
  journal      = {Operations Research},
  volume       = {69},
  number       = {1},
  pages        = {82--99},
  year         = {2021}
}

@article{eor/SmithS20,
  author       = {J. Cole Smith and
                  Yongjia Song},
  title        = {A survey of network interdiction models and algorithms},
  journal      = {European Journal of Operations Research},
  volume       = {283},
  number       = {3},
  pages        = {797--811},
  year         = {2020}
}

@article{edmonds1970,
  title={Submodular functions, matroids and certain polyhedra},
  author={Edmonds, Jack},
  journal={Combinatorial Structures and Their Applications},
  pages={69--87},
  year={1970}
}

@article{cunningham1984testing,
  title={Testing membership in matroid polyhedra},
  author={Cunningham, William H},
  journal={Journal of Combinatorial Theory, Series B},
  volume={36},
  number={2},
  pages={161--188},
  year={1984},
  publisher={Elsevier}
}

@article{ageev2004pipage,
  title={Pipage rounding: A new method of constructing algorithms with proven performance guarantee},
  author={Ageev, Alexander A and Sviridenko, Maxim I},
  journal={Journal of Combinatorial Optimization},
  volume={8},
  number={3},
  pages={307--328},
  year={2004},
  publisher={Springer}
}

@article{nemhauser1978analysis,
  title={An analysis of approximations for maximizing submodular set functions—{I}},
  author={Nemhauser, George L and Wolsey, Laurence A and Fisher, Marshall L},
  journal={Mathematical Programming},
  volume={14},
  pages={265--294},
  year={1978},
  publisher={Springer}
}

@article{ketkov2025class,
  title={On a class of interdiction problems with partition matroids: complexity and polynomial-time algorithms},
  author={Ketkov, Sergey S and Prokopyev, Oleg A},
  journal={INFORMS Journal on Computing},
  year={2025},
  publisher={INFORMS}
}

@article{NemhauserW78,
  author       = {George L. Nemhauser and
                  Laurence A. Wolsey},
  title        = {Best Algorithms for Approximating the Maximum of a Submodular Set
                  Function},
  journal      = {Mathematics of  Operations Research},
  volume       = {3},
  number       = {3},
  pages        = {177--188},
  year         = {1978}
}

@article{HS23,
title = {Matroid-constrained vertex cover},
journal = {Theoretical Computer Science},
volume = {965},
pages = {113977},
year = {2023},
issn = {0304-3975},
author = {Chien-Chung Huang and François Sellier},
}

@book{dempe2002foundations,
  title={Foundations of bilevel programming},
  author={Dempe, Stephan},
  year={2002},
  publisher={Springer}
}

@article{colson2007overview,
  title={An overview of bilevel optimization},
  author={Colson, Beno{\^\i}t and Marcotte, Patrice and Savard, Gilles},
  journal={Annals of Operations Research},
  volume={153},
  number={1},
  pages={235--256},
  year={2007},
  publisher={Springer}
}

@article{israeli2002shortest,
  title={Shortest-path network interdiction},
  author={Israeli, Eitan and Wood, R Kevin},
  journal={Networks: An International Journal},
  volume={40},
  number={2},
  pages={97--111},
  year={2002},
  publisher={Wiley Online Library}
}

@article{wood1993deterministic,
  title={Deterministic network interdiction},
  author={Wood, R Kevin},
  journal={Mathematical and Computer Modelling},
  volume={17},
  number={2},
  pages={1--18},
  year={1993},
  publisher={Elsevier}
}

@inproceedings{linhares2017improved,
  title={Improved Algorithms for MST and Metric-TSP Interdiction},
  author={Linhares, Andr{\'e} and Swamy, Chaitanya},
  booktitle={44th International Colloquium on Automata, Languages, and Programming (ICALP)},
  volume={80},
  pages={32},
  year={2017}
}

\end{document}